\documentclass[11pt]{article}
\usepackage[margin=1in]{geometry}
\usepackage{amsmath,amssymb,amsthm,mathtools}
\usepackage[T1]{fontenc}
\usepackage{lmodern}
\usepackage{needspace}
\usepackage[colorlinks=true,linkcolor=blue,citecolor=blue,urlcolor=blue]{hyperref}
\newtheorem{theorem}{Theorem}[section]
\newtheorem{lemma}[theorem]{Lemma}
\newtheorem{proposition}[theorem]{Proposition}
\newtheorem{corollary}[theorem]{Corollary}
\newtheorem{conjecture}[theorem]{Conjecture}
\newtheorem{maintheorem}{Theorem}

\newtheorem{goalconjecture}{Conjecture}

\theoremstyle{remark}
\newcommand{\Z}{\mathbb Z}
\newcommand{\Q}{\mathbb Q}
\newcommand{\C}{\mathbb C}
\newcommand{\Fp}{\mathbb F_p}
\newcommand{\Res}{\operatorname{Res}}
\newcommand{\Norm}{\operatorname{Norm}}
\newcommand{\logplus}{\log^+}
\title{Small sums of roots of unity:\\polylogarithmic exceptional sets and conditional separation
  \footnote{AI disclosure: This paper was written with consultation to Large Language Models such as ChatGPT 6 and Claude 5. The author takes full responsibility of the mathematical content.}}
\author{Qi Cheng
  \footnote{School of Computer Science, University of Oklahoma. {\tt Email:qchengnorman@outlook.com}}
   }
\date{October 16, 2026}
\begin{document}
\maketitle
\begin{abstract}
We prove a polylogarithmic exceptional-set estimate for symmetric five-term sums of roots of unity of arbitrary order. More precisely, at most $O((\log(2N))^{11})$ points of the $N$-grid satisfy
$|1+2\cos(2\pi x)+2\cos(2\pi y)|<\exp(-C(\log(2N))^{13})$.
This implies the corresponding bound for small conjugates at the reduced conductor: outside polylogarithmically many real embeddings, every such sum is at least quasipolynomially large. The proof is unconditional and combines a polynomial-vector-field multiplicity estimate, arithmetic jet determinants, analytic propagation, and interpolation determinants for nearby grid points. As a companion result, assuming a uniform number-field $abc$ conjecture, we prove that every sum of a fixed number of distinct prime-order roots of unity has modulus at least $\exp(-\varepsilon p)$ for every $\varepsilon>0$ and all sufficiently large primes $p$, uniformly in the exponents. That proof uses simultaneous compression of the exponents, discriminant estimates, and cyclic resultants.
\end{abstract}

\section{Introduction}
How small can a nonzero sum of a few roots of unity be? For integers $k,N\ge2$, this asks for lower bounds on
\[
 \left|\zeta_N^{a_1}+\cdots+\zeta_N^{a_k}\right|,
 \qquad \zeta_N=\exp(2\pi i/N),
\]
uniformly in the exponents, subject to nonvanishing. Myerson formulated the minimum-modulus problem in \cite{Myerson}. The distinction between exact cancellation and very close cancellation is essential: algebraic relations explain when a sum vanishes, but need not give strong quantitative separation for sums which do not vanish.

An elementary norm argument illustrates the difficulty. If $\alpha$ is a nonzero sum of $k$ $N$th roots, it is an algebraic integer in $\Q(\zeta_N)$, every conjugate has modulus at most $k$, and its nonzero norm has absolute value at least $1$. Consequently
\[
 |\alpha|\ge k^{-(\varphi(N)-1)}.
\]
For fixed $k$, this is exponential in the field degree. The expression defining $\alpha$, however, uses only $k$ exponents of bit length $O(\log N)$. A bound such as
$\exp\{-C_k(\log N)^{B_k}\}$ would be much closer to the size of this sparse description. Even replacing the exponential-in-$p$ bound by $\exp(-o(p))$, uniformly at prime orders, requires additional information.

\subsection{Motivation from numerical separation and identity testing}
The sum-of-square-roots problem provides a useful model for the separation questions studied here. Given positive integers $n_1,\ldots,n_s\le H$ and signs $\epsilon_j\in\{1,-1\}$, one seeks to decide the sign of
\[
 S=\sum_{j=1}^s\epsilon_j\sqrt{n_j}.
\]
Such comparisons arise when comparing lengths of polygonal paths with integer coordinates: each Euclidean segment length is a square root of an integer. Numerical approximation can certify the sign of a nonzero $S$ once the approximation error is smaller than $|S|$, so the required precision is governed by a separation bound.

The elementary algebraic estimate can be very weak. The field generated by the square roots has degree at most $2^s$, and every conjugate of $S$ has absolute value at most $s\sqrt H$. If $S\ne0$, integrality of its norm therefore gives
\[
 |S|\ge(s\sqrt H)^{-(2^s-1)}.
\]
The resulting precision bound is exponential in the number of terms, despite the short description of the input. Improving this dependence is a central issue in the sum-of-square-roots problem. Cheng \cite{ChengSqrt} obtains sharper separation estimates when the radicands are small relative to the number of terms, leading to a subexponential comparison algorithm in that regime.

Short sums of roots of unity exhibit the same tension between a succinct expression and a large ambient field. Here the growing parameter is the order $N$, even when the number of summands is fixed. A lower bound $\exp\{-C(\log N)^B\}$ would permit certification of nonvanishing with polynomially many bits in $\log N$ for a fixed-length sum. This is why the distinction between exponential, subexponential, and quasipolynomial separation is computationally meaningful. The comparison with square roots is motivational; the arithmetic structure of the two families is different.

A second motivation is polynomial identity testing by approximate algebraic evaluation. Chen--Kao \cite{ChenKao} evaluate integer polynomials at suitable algebraic points constructed from signed sums of square roots of distinct primes. Degree restrictions ensure that an identically nonzero polynomial remains nonzero at the exact algebraic point. Their randomized choice of conjugates, together with a bound on exceptionally small values, allows the error probability to be decreased by increasing numerical precision without increasing the number of random bits. A cyclotomic version of this approach evaluates polynomials at roots of unity and asks for analogous control of small conjugates. For the algorithmic problem of deciding exact vanishing at roots of unity, see also Balaji--Perifel--Shirmohammadi--Worrell \cite{CIT}.

The role of an exceptional-set bound in removing randomness is particularly transparent. Suppose that, for a prescribed class of $P\in\Z[X]$, one knows uniformly that
\[
 \#\{t\in(\Z/N\Z)^\times:|P(\zeta_N^t)|<\delta_N\}\le B_N
 \qquad\text{whenever }P(\zeta_N)\ne0.
\]
If $\varphi(N)>B_N$, any $\lfloor B_N\rfloor+1$ distinct admissible indices contain a value of modulus at least $\delta_N$; there is no need to identify the exceptional indices or to choose an index randomly. Approximating each value with certified absolute error at most $\delta_N/4$ and testing whether its computed modulus exceeds $\delta_N/2$ distinguishes the nonzero case from the case in which all exact values vanish. When $B_N$ and $\log(1/\delta_N)$ are polynomial in $\log N$, this requires only polylogarithmically many evaluations and polynomial output precision.

For this to be an identity test rather than merely a test modulo the cyclotomic polynomial, one also needs $P\not\equiv0\Rightarrow P(\zeta_N)\ne0$; the condition $\deg P<\varphi(N)$ is one sufficient guarantee. An algorithmic conclusion further requires effective bounds and efficient construction and certified evaluation of the chosen points in the input representation. Theorem~\ref{thm:D} below supplies the exceptional-set phenomenon for the specific symmetric five-term family (with indices taken modulo sign). Extending it to suitable polynomial classes would provide a route to deterministic numerical identity testing in the spirit of Chen--Kao; the theorem does not by itself give a general derandomization of polynomial identity testing.

Finally, small root-of-unity sums also arise in counting solutions of linear equations whose variables lie in a fixed subset of a finite field. Zhuang--Cheng--Wen \cite{ZCW} relate these counts to Fourier coefficients given by short sums of roots of unity, making separation bounds relevant to the resulting counting errors.

\subsection{Previous work and the five-term case}
Konyagin--Lev \cite{KL} used the distribution of exponential sums to strengthen lower bounds for sums of distinct roots of prime order. Their work adds Fourier information to the norm argument, while leaving the problem of uniform subexponential separation. Habegger \cite[Theorem~7]{Hab} obtained polynomial lower bounds for fixed-length sums simultaneously outside a sparse set of prime orders. This is a strong almost-all-primes result; the problem considered here also includes the exceptional primes.

Constructions of small sums address the opposite side of the question. The cases with at most four terms admit explicit geometric analysis, whereas five terms already exhibit more complicated cancellation. Barber \cite{Barber} proved upper bounds of order $N^{-4/3}$ for five-term sums, with an improvement to $N^{-7/3}$ for infinitely many orders. Kalogirou \cite{Kalogirou} gives further constructions for general numbers of terms. These results indicate the richness of near-cancellation configurations; they do not supply uniform lower bounds.

A particularly natural five-term family is
\[
 1+\zeta_N^a+\zeta_N^{-a}+\zeta_N^b+\zeta_N^{-b}
 =1+2\cos(2\pi a/N)+2\cos(2\pi b/N).
\]
The sum is real, and its small values correspond to rational grid points close to a fixed analytic curve. This geometric description suggests counting the exceptionally small conjugates even when a lower bound at every conjugate is unavailable. In a complementary direction, cyclic resultants and Mahler measure relate products over conjugates to arithmetic growth. Stewart's work on exceptional units and cyclic resultants \cite{Stewart}, and Dimitrov--Habegger's logarithmic equidistribution results for torsion points near atoral sets \cite{DH}, supply relevant arithmetic context.

The present paper develops two conclusions, at different levels of certainty. Unconditionally, Theorem~\ref{thm:D} gives quasipolynomial separation outside polylogarithmically many embeddings for the symmetric five-term family at arbitrary reduced conductor. As a companion, under the uniform number-field $abc$ conjecture, formulated precisely in Section~\ref{sec:abc}, simultaneous compression and a cyclic-resultant estimate give subexponential separation for every fixed-length sum of distinct prime-order roots. The first conclusion bounds an exceptional set unconditionally; the second controls every conjugate under an arithmetic hypothesis.

\subsection{The separation goals, the main result, and a companion result}\label{sec:results}

To fix the targets, we state two goals for symmetric five-term sums at prime orders. For a prime $p>5$, call $(a,b)\in\Fp^2$ admissible if $a,b\ne0$ and $a\ne\pm b$, and put
\[
 \alpha_p(a,b)=1+\zeta_p^a+\zeta_p^{-a}+\zeta_p^b+\zeta_p^{-b},
 \qquad\zeta_p=\exp(2\pi i/p).
\]
The five roots are then distinct, and $\alpha_p(a,b)\ne0$. Indeed, for a prime $p>k$, a sum of $k$ distinct $p$th roots of unity cannot vanish: its exponent polynomial, of degree at most $p-1$ with coefficients $0$ or $1$, could only vanish at a primitive $p$th root by being a multiple of $\Phi_p=1+X+\cdots+X^{p-1}$, and that would require all $p$ terms. Both conjectures below are unconditional assertions, with constants uniform over the admissible pairs.

\begin{goalconjecture}[Uniform quasipolynomial separation]\label{conj:A}
There are absolute constants $C>0$ and $B\ge1$ such that, for every prime $p>5$ and every admissible pair $(a,b)$,
\[
 |\alpha_p(a,b)|\ge \exp\{-C(\log p)^B\}.
\]
\end{goalconjecture}

\begin{goalconjecture}[Uniform subexponential separation]\label{conj:B}
For every $\varepsilon>0$ there is $p_0(\varepsilon)>5$ such that, for every prime $p\ge p_0(\varepsilon)$ and every admissible pair $(a,b)$,
\[
 |\alpha_p(a,b)|\ge \exp(-\varepsilon p).
\]
Equivalently, the minimum over all admissible pairs is at least $\exp(-o(p))$ as $p$ tends to infinity through primes.
\end{goalconjecture}

Conjecture~\ref{conj:A} implies Conjecture~\ref{conj:B}, since $(\log p)^B=o(p)$ for each fixed $B$. Neither is proved in this paper. Our main result gives the separation scale of Conjecture~\ref{conj:A}, with exponent $13$, at every embedding outside a polylogarithmic exceptional set, and it does so unconditionally and at arbitrary conductor.

\Needspace{13\baselineskip}
\begin{maintheorem}\label{thm:D}
There are absolute constants $A,C>0$ with the following property. Let $N\ge3$, let $a,b\in\mathbb Z$ satisfy $\gcd(a,b,N)=1$, and suppose
\[
 \alpha=1+\zeta_N^a+\zeta_N^{-a}+\zeta_N^b+\zeta_N^{-b}\ne0,
 \qquad \zeta_N=e^{2\pi i/N}.
\]
For $K_N=\mathbb Q(\zeta_N)^+$, the number of embeddings
$\sigma:K_N\hookrightarrow\mathbb R$ for which
\[
 |\sigma(\alpha)|<\exp\{-C(\log(2N))^{13}\}
\]
is at most $A(\log(2N))^{11}$.
\end{maintheorem}
Theorem~\ref{thm:D} is unconditional. In particular it applies to five distinct roots with
$0<a<b<N/2$, but distinctness is not needed for the assertion as stated. For a nonprimitive pair of exponents one first passes to the reduced conductor; counting embeddings in a larger cyclotomic field introduces a multiplicity factor. Section~\ref{sec:Dproof} proves a stronger full-grid estimate and explains this qualification. For prime $N=p>5$ and admissible $(a,b)$, the theorem says that at most $A(\log(2p))^{11}$ of the $(p-1)/2$ conjugates of $\alpha_p(a,b)$ are smaller than $\exp\{-C(\log(2p))^{13}\}$; it leaves those exceptional conjugates uncontrolled.

There are related but different counting results in the literature. Habegger's approximation theorem \cite[Corollary~3]{Hab} gives subpolynomial bounds for rational points sufficiently close to definable sets, with the appropriate algebraic-part exceptions. Binyamini--Novikov prove Wilkie's conjecture for restricted elementary functions \cite{BN}, concerning exact rational points on the transcendental part. Neither statement, by itself, gives the explicit polylogarithmic count in the shrinking neighborhood considered here. The proof in Section~\ref{sec:Dproof} uses the established multiplicity and zero-counting framework of Gabrielov, Comte--Yomdin, and Binyamini \cite{Gabrielov,CY,Binyamini}; its specific ingredients are verified for this curve. We make no general novelty claim for the underlying method.

The companion result controls every conjugate, at the weaker subexponential scale of Conjecture~\ref{conj:B}, under an arithmetic hypothesis.

\begin{theorem}[Conditional subexponential separation]\label{thm:main}
Assume the uniform number-field $abc$ conjecture in the form stated in Section~\ref{sec:abc}. For every fixed integer $k\ge2$ and every $\varepsilon>0$, there is $p_0(k,\varepsilon)$ such that, for every prime $p\ge p_0(k,\varepsilon)$ and every set of pairwise distinct exponents $a_1,\ldots,a_k\in\Fp$,
\[
\left|\sum_{j=1}^k\zeta_p^{a_j}\right|\ge\exp(-\varepsilon p).
\]
The threshold is independent of the exponents. For $k=1$ the modulus is $1$.
\end{theorem}
The sum is nonzero for $p>k$ by the cyclotomic-polynomial argument above. The conjectural hypothesis is uniform over number fields of \emph{unbounded degree}. Ordinary $abc$ over $\Q$, or number-field $abc$ with arbitrary field-dependent constants, is not the hypothesis used here. The conclusion is $\exp(-o(p))$ in the precise quantified sense of Theorem~\ref{thm:main}. We do not deduce an explicit exponent $p^{1-\delta}$ or a quasipolynomial lower bound. Distinctness, positive unit coefficients, and prime order are the scope of the theorem as written.

In terms of the goals: Theorem~\ref{thm:main} proves Conjecture~\ref{conj:B} under uniform number-field $abc$, and extends that conditional statement to every fixed number of distinct prime-order terms; it does not prove Conjecture~\ref{conj:A}. Theorem~\ref{thm:D} gives the scale of Conjecture~\ref{conj:A} outside an exceptional set, but says nothing about the exceptional embeddings. Thus neither unconditional conjecture follows from the results of this paper. The remaining goals are to eliminate the exceptional embeddings in Theorem~\ref{thm:D}, or to remove the arithmetic hypothesis in Theorem~\ref{thm:main}.

The paper is organized as follows. Section~\ref{sec:Dproof} proves Theorem~\ref{thm:D}. Sections~\ref{sec:abc}--\ref{sec:fiveterm} prove Theorem~\ref{thm:main} and record a stronger aggregate consequence of the same hypothesis. Section~\ref{sec:limits} discusses limitations and the relation to earlier approaches. The two parts are logically independent and use disjoint notation; the only constants shared between them are the generic $C$ and $c$, whose values may change between occurrences.

\section{An unconditional polylogarithmic exceptional-set bound}\label{sec:Dproof}
Put
\[
 f(x,y)=1+2\cos(2\pi x)+2\cos(2\pi y),\qquad
 X=\{(x,y)\in(\mathbb R/\mathbb Z)^2:f(x,y)=0\}.
\]
All constants in this section are absolute, unless their dependence on a fixed chart is indicated. They may increase between occurrences. The exponents below are deliberately nonoptimal.

\begin{theorem}[Grid form of Theorem~\ref{thm:D}]\label{thm:grid}
There exist constants $A,C>0$ such that, for every integer $N\ge2$,
\[
 \#\left\{(i,j)\in\{0,\ldots,N-1\}^2:
 |f(i/N,j/N)|<\exp\{-C(\log(2N))^{13}\}\right\}
 \le A(\log(2N))^{11}.
\]
The counted set includes points at which $f$ vanishes exactly.
\end{theorem}

The argument combines a multiplicity estimate of Gabrielov \cite{Gabrielov}, arithmetic lower bounds for jet determinants, analytic propagation, and interpolation determinants. The general connection between multiplicity, Bernstein estimates, and rational-point counting is developed in \cite{CY,Binyamini}. The determinant method originates in work such as \cite{BP}. We give the additional estimates for points \emph{near} $X$, since an exact rational-point theorem alone would not imply Theorem~\ref{thm:grid}.

\subsection{Logarithmic lifts and a uniform jet bound}

Consider the affine curve
\[
 E=\{(z,w)\in(\C^\times)^2:1+z+z^{-1}+w+w^{-1}=0\}.
\]
It is smooth: a singular point would have $z,w\in\{1,-1\}$, whereas the defining expression there is $5$, $1$, or $-3$. It is irreducible. Indeed, its equation as a quadratic in $w$ has discriminant
\[
 (z^2+z+1)^2-4z^2=(z^2-z+1)(z^2+3z+1),
\]
up to the square factor $z^{-2}$; the four roots of the displayed polynomial are distinct, so the discriminant is not a square in $\C(z)$.

\begin{lemma}\label{lem:logind}
No nonzero polynomial in two logarithmic coordinates vanishes identically on a local logarithmic lift of $E$.
\end{lemma}
\begin{proof}
Above $z=\infty$, the two branches of the normalization have
$w=-z(1+o(1))$ and $w=-z^{-1}(1+o(1))$. Small loops about these punctures, joined to a fixed base point by paths in the smooth connected curve, give logarithmic increments
$2\pi i(1,1)$ and $2\pi i(1,-1)$, with a possible simultaneous change of orientation. If a polynomial vanished on a local lift, analytic continuation would make it vanish at every point of a translate of the lattice generated by these two vectors. After an invertible affine change of coordinates it would vanish on $\mathbb Z^2$. Successive application of the fact that a one-variable polynomial has finitely many roots shows that it is the zero polynomial.
\end{proof}

Set $\omega=e^{\pi i/3}$ and take the base point $(z,w)=(-1,\omega)$.
Use centered coordinates
\[
 U=\log(z/(-1)),\qquad V=\log(w/\omega).
\]
On the logarithmic lift the following polynomial vector field is tangent:
\[
 \mathcal D=z(w^2-1)\partial_U-w(z^2-1)\partial_V
       +z^2(w^2-1)\partial_z-w^2(z^2-1)\partial_w .
\]
Indeed it annihilates the defining Laurent polynomial of $E$, and its $U,V$ components are respectively $\mathcal D z/z,\mathcal D w/w$. At
$b=(0,0,-1,\omega)$ the vector field is nonzero. Its degree is $4$.

For $d\ge1$ put
\[
 M_d=\binom{d+2}{2},\qquad
 L_d=512(\max\{d,3\}+9)^8,\qquad q_d=(d+1)^{11}.
\]
Gabrielov's multiplicity estimate \cite[Theorem~2]{Gabrielov}, applied in ambient dimension $4$ to a degree-$4$ vector field, shows that a polynomial of degree at most $d$ which is not identically zero on the trajectory has order of vanishing less than $L_d$. Here one may use degree $\max(d,3)$ in that estimate. By Lemma~\ref{lem:logind}, the matrix
\[
 \mathcal J_d=\bigl(\mathcal D^k(U^iV^j)(b)\bigr)_{
       0\le k\le L_d,\ i+j\le d}
\]
therefore has column rank $M_d$.

All entries belong to $\mathbb Z[\omega]$. A nonzero square minor has absolute value at least $1$: its squared absolute value is its positive integral norm in $\mathbb Q(\omega)$. Also, for an absolute constant $C_1$,
\[
 |\mathcal D^k(U^iV^j)(b)|
 \le H_d:=\{C_1(d+3L_d+1)\}^{L_d}.
\]
To see this, each application of $\mathcal D$ increases degree by at most $3$ and multiplies the coefficient $\ell^1$ norm by at most a constant times the current degree; evaluation at $b$ does not increase that norm. Choose an invertible $M_d$-row minor and apply the cofactor formula to its inverse. For every $Q\in\C[U,V]$ of degree at most $d$,
\begin{equation}\label{eq:jetlower}
 \max_{0\le k\le L_d}|(\mathcal D^kQ)(b)|
 \ge \frac{\|Q\|}{M_d!H_d^{M_d-1}}
 \ge \|Q\|\exp\{-C_2(d+1)^{10}\log(d+2)\},
\end{equation}
where $\|Q\|$ denotes the maximum absolute coefficient.

Let $\gamma(\tau)$ be the holomorphic trajectory through $b$. On a sufficiently small fixed disc in $\tau$, Cauchy's inequalities and
$(Q\circ\gamma)^{(k)}(0)=(\mathcal D^kQ)(b)$ turn
\eqref{eq:jetlower} into a lower bound for $\sup|Q\circ\gamma|$ with the same exponent after enlarging $C_2$. The factors $k!r^{-k}$ cost only $O(L_d\log(L_d+2))$ in the logarithm. Since $\mathcal D U(b)\ne0$, one may also use $U$ as the local parameter.

\subsection{Analytic estimates on real graph charts}

The real gradient of $f$ does not vanish on $X$: its vanishing would again force both cosines to be $\pm1$. By compactness, $X$ has a finite cover by real graph charts, each of the form $y=g(x)$ after possibly interchanging coordinates and choosing real lifts modulo $\mathbb Z$. The functions $g$ extend holomorphically to larger complex discs. We shrink the real intervals and their complex neighborhoods so that every chart has fixed positive margins.

\begin{lemma}\label{lem:bernstein}
For each of these charts there are a real interval $I$, concentric complex discs $D_{\rm in}\subset D_{\rm out}$, with $I\subset D_{\rm in}$ and with fixed positive margins, and a constant $C$, such that
\[
 \sup_{D_{\rm in}}|P(x,g(x))|\ge \|P\|e^{-Cq_d},
 \qquad
 \sup_{D_{\rm out}}|P(x,g(x))|\le \|P\|e^{C(d+1)}
\]
for every polynomial $P\in\C[x,y]$ of degree at most $d$. The inner disc may be replaced by any prescribed smaller fixed disc in the chart, at the cost of changing $C$.
\end{lemma}
\begin{proof}
At the base point substitute
\[
 x=\tfrac12+U/(2\pi i),\qquad y=\tfrac16+V/(2\pi i).
\]
This affine substitution and its inverse change the coefficient norm by factors bounded by $e^{C(d+1)}$. Thus \eqref{eq:jetlower} and Cauchy's inequalities give the asserted lower bound on a fixed base disc, since
$(d+1)^{10}\log(d+2)\le Cq_d$.

Connect the base disc to each target chart by a fixed path in $E$, and continue the logarithmic coordinates along it. The coordinates are bounded on a finite chain of relatively compact discs covering the lifted path, so the upper bound $e^{C(d+1)}\|P\|$ holds throughout the chain. Hadamard's three-circle inequality propagates a lower bound from one disc to an overlapping smaller disc: if the upper bound is $B^\ast$ and the initial lower bound is $b^\ast$, the new lower bound has the form $B^\ast(b^\ast/B^\ast)^{c}$ for a fixed $c\ge1$. Iterating a fixed finite number of times preserves the exponent $Cq_d$. Integer shifts in the terminal logarithmic coordinates merely make a fixed affine substitution in $P$, which again costs $e^{C(d+1)}$. Interchanging the coordinates has no cost. The direct bound for a polynomial on a bounded complex set gives the upper estimate. The same propagation permits further fixed shrinking.
\end{proof}

\begin{lemma}[Intervals and components]\label{lem:interval}
On each fixed chart the following hold, for $d\ge1$.
\begin{enumerate}
\item If $J\subset I$ is a closed interval of length $0<\ell\le1$, then
\[
 \sup_J|P(x,g(x))|
 \ge \|P\|\exp\{-Cq_d(1+\log(1/\ell))\}.
\]
\item If $P$ has real coefficients, $\|P\|=1$, and $0<\varepsilon<1/2$, then
\[
 \{x\in I:|P(x,g(x))|<\varepsilon\}
\]
has at most $Cq_d$ connected components.
\end{enumerate}
\end{lemma}
\begin{proof}
Write $h(x)=P(x,g(x))$. For the first assertion use a fixed disc containing $D_{\rm in}$ with closure inside $D_{\rm out}$, and slit it along $J$. The harmonic measure of the slit, at points of $D_{\rm in}\setminus J$, is at least
$c/(1+\log(1/\ell))$. Here is one way to obtain a bound uniform in $J$. Integrate the Green function of the disc against the logarithmic equilibrium probability measure of $J$. On the fixed inner compact set this positive potential is bounded below by a positive constant. On $J$ it is bounded above by $C+\log(1/\ell)$, since the logarithmic capacity of an interval of length $\ell$ is $\ell/4$ and the regular part of the Green function is bounded on the compact set. It vanishes on the outer circle. The maximum principle compares the potential divided by $C+\log(1/\ell)$ with harmonic measure. These Green-function and two-constants principles are standard; see \cite{Ransford}.

Apply the two-constants principle to the subharmonic function $\log|h|$, using its upper bound on the outer disc and its supremum on $J$. At a point of $D_{\rm in}$ where $|h|$ has the lower bound of Lemma~\ref{lem:bernstein}, this gives the assertion. If that point lies on $J$, the assertion follows directly. Zeros are handled by the usual limiting form of the subharmonic maximum principle.

For the second assertion apply Lemma~\ref{lem:bernstein} to $P-\varepsilon$ and $P+\varepsilon$, whose coefficient norms lie between $1/2$ and $3/2$. Jensen's formula bounds their numbers of zeros near $I$ by $Cq_d$. More explicitly, use a point attaining the lower estimate on a sufficiently small fixed disc as the center of two discs of fixed radii: the smaller contains the real interval under consideration and the larger stays in $D_{\rm out}$. Jensen bounds the number by a constant times the logarithm of the ratio of the outer supremum to the value at the center. A fixed subdivision of $I$, if necessary, gives the stated bound on all of $I$. Neither function vanishes identically, by Lemma~\ref{lem:logind}. Each interior boundary point of the sublevel set is a zero of one of these functions; hence the number of components is at most their total number of real zeros plus one.
\end{proof}

\subsection{An interpolation determinant for nearby grid points}

Compactness and the nonvanishing gradient imply that, for all sufficiently small $\delta>0$, the set $|f|<\delta$ is covered by finitely many of the preceding graph neighborhoods, and in each neighborhood
\begin{equation}\label{eq:tube}
 |y-g(x)|\le C\delta.
\end{equation}
Coordinates are real lifts, possibly interchanged. Lifting a grid point preserves the property that both coordinates belong to $N^{-1}\mathbb Z$. Subdivide the real chart intervals into a fixed finite number of intervals of length at most $\ell_0$, chosen below.

Fix $d\ge1$ and let $\nu=M_d$. For any $\nu$ grid points in one such interval form the evaluation determinant $\Delta$ of the monomials $x^ay^b$, $a+b\le d$. Put
\[
 S_d=\sum_{a+b\le d}(a+b)=\frac{d(d+1)(d+2)}3 .
\]
Then
\begin{equation}\label{eq:denominator}
 \Delta=0\quad\hbox{or}\quad |\Delta|\ge N^{-S_d}.
\end{equation}
Indeed multiplying the column for $x^ay^b$ by $N^{a+b}$ gives an integer matrix.

Move each point vertically onto the graph and denote the resulting determinant by $\Delta_X$. The monomials restricted to the graph are bounded by $\Theta^{d+1}$ on a fixed complex neighborhood, for a fixed $\Theta\ge1$. Newton divided differences in the $x$ coordinates give
\[
 |\Delta_X|
 \le \nu!\, C^{\nu}\, \Theta^{(d+1)\nu}
       (\ell_0/r_0)^{\nu(\nu-1)/2},
\]
where $r_0>0$ is a fixed radius. To justify this estimate, the row operations extract the Vandermonde product, bounded by $\ell_0^{\nu(\nu-1)/2}$; Cauchy's integral formula on a fixed surrounding contour of radius $r_0$ bounds a divided difference of order $m$ by $C\Theta^{d+1}r_0^{-m}$. The formula extends to repeated nodes by continuity. Choose $\ell_0<r_0$ sufficiently small. Thus
\[
 |\Delta_X|\le \exp\{-c\nu^2+Cd\nu+C\nu\log(\nu+1)\}.
\]
On the same compact neighborhood the monomials and their first derivatives are bounded by $e^{C(d+1)}$. Multilinearity of the determinant and \eqref{eq:tube} consequently give
\begin{equation}\label{eq:neardet}
 |\Delta|
 \le e^{-c\nu^2+Cd\nu+C\nu\log(\nu+1)}
       +\delta e^{Cd\nu+C\nu\log(\nu+1)}.
\end{equation}
The constants are uniform over our finite chart cover.

\begin{proof}[Proof of Theorem~\ref{thm:grid}]
Write $L=\log(2N)$ and choose $d=\lceil\kappa L\rceil$, with a sufficiently large fixed $\kappa$. Since $\nu\asymp d^2$ and $S_d\asymp d^3$, the first term of \eqref{eq:neardet} is less than $\tfrac14N^{-S_d}$ for all sufficiently large $N$. Take
$\delta=\exp(-C_0L^{13})$. For a sufficiently large fixed $C_0$, the second term is also less than $\tfrac14N^{-S_d}$. By \eqref{eq:denominator}, every $\nu$-row evaluation determinant in a chart interval is zero.

If the interval contains fewer than $\nu$ points, it contributes $O(L^2)$ and there is nothing to prove. Otherwise its evaluation matrix has rank less than $\nu$, so there is a nonzero real polynomial $P$ of degree at most $d$ vanishing at all its grid points; the kernel vector can be taken real because the matrix is real. Normalize $\|P\|=1$. By \eqref{eq:tube} and the derivative bound,
\[
 |P(x,g(x))|<\varepsilon:=2\delta e^{C(d+1)}
\]
at their abscissae. For large $N$, $\varepsilon<1/2$. Lemma~\ref{lem:interval} bounds the number of components of this sublevel set by $Cq_d=O(L^{11})$.

No such component can have length greater than $1/N$. If it did, it would contain a closed interval $J$ of length $1/N$, and Lemma~\ref{lem:interval} would give
\[
 \sup_J|P(x,g(x))|
 \ge \exp\{-Cq_d(1+\log N)\}\ge \exp(-C'L^{12}),
\]
contradicting $\sup_J|P(x,g(x))|\le\varepsilon$ for large $N$.
Each component therefore contains at most two elements of $N^{-1}\mathbb Z$. For each abscissa there is at most one grid ordinate satisfying \eqref{eq:tube}, since $2C\delta<1/N$. This proves an $O(L^{11})$ bound in the chart interval. Sum over the fixed finite cover, then enlarge $A$ to include the finitely many remaining $N$. This proves the theorem.
\end{proof}

\subsection{Conjugates, reduced conductors, and nonvanishing}

\begin{proof}[Proof of Theorem~\ref{thm:D}]
Let $N\ge3$, $\gcd(a,b,N)=1$, and
$\alpha=1+\zeta_N^a+\zeta_N^{-a}+\zeta_N^b+\zeta_N^{-b}$.
For $t\in(\mathbb Z/N\mathbb Z)^\times$,
\[
 \sigma_t(\alpha)=f(ta/N,tb/N).
\]
The map $t\mapsto(ta,tb)\pmod N$ is injective: if both differences vanish, an integral linear combination of $a,b,N$ equal to $1$ shows that $t-t'\equiv0\pmod N$. The grid bound therefore bounds the number of such $t$, and hence also the number of classes modulo $t\sim-t$. These classes are precisely the embeddings of $\mathbb Q(\zeta_N)^+$ into $\mathbb R$. If $\alpha\ne0$, none of its conjugates is zero. Theorem~\ref{thm:grid} gives the assertion.
\end{proof}

For arbitrary exponents put $g_0=\gcd(a,b,N)$, $N_0=N/g_0$, and
$(a',b')=(a/g_0,b/g_0)$. Then $\alpha_N(a,b)=\alpha_{N_0}(a',b')$, and the preceding bound applies at the reduced conductor $N_0$ when $N_0\ge3$. If embeddings of the larger real cyclotomic field are counted instead, each reduced-conductor embedding occurs $\varphi(N)/\varphi(N_0)$ times (for $N_0,N\ge3$); this multiplicity must not be discarded. Conductors $1$ and $2$ give integer sums and can be treated directly.

The nonzero hypothesis is natural for the conjugate formulation. It cannot be omitted uniformly for all conductors: the five fifth roots have sum zero. For prime $N>5$ distinctness ensures nonvanishing by the cyclotomic-polynomial argument in Section~\ref{sec:results}. For $N$ a power of $2$, a sum of five roots cannot vanish: reduction modulo a prime above $2$ sends every $N$th root to $1$, and hence sends the sum to $1$. The proof of the grid theorem itself does not require a nonvanishing hypothesis or a prime-power conductor.

Theorem~\ref{thm:D} concerns an exceptional set of embeddings, not a lower bound at every embedding. It therefore does not supply an unconditional replacement for Theorem~\ref{thm:main}, to which we now turn. The special five-term symmetric expression enters through the fixed curve $E$, its logarithmic monodromy, and the two-dimensional grid determinant. No assertion for arbitrary fixed-length sums is made here.

\section{Heights, radicals, and the uniform hypothesis}\label{sec:abc}
The remainder of the paper proves Theorem~\ref{thm:main}. The argument is arithmetic rather than analytic, and the notation of this and the following sections is independent of Section~\ref{sec:Dproof}.

For a number field $K$ of degree $e$, write $\Delta_K$ for its absolute field discriminant. For a nonzero triple $a,b,c\in K$, use the absolute logarithmic projective height
\[
h(a:b:c)=\frac1e\sum_w [K_w:\Q_w]\log\max\{|a|_w,|b|_w,|c|_w\},
\]
with the usual absolute values and local-degree normalization; equivalently one may sum over prime ideals and all complex embeddings, counting conjugate complex embeddings separately. The product formula makes the height invariant under scaling. Define the finite radical
\[
r_K(a:b:c)=\frac1e\sum_{\mathfrak p\in\mathcal S(a,b,c)}\log\Norm\mathfrak p,
\]
where $\mathcal S(a,b,c)$ consists of the nonzero prime ideals at which the valuations of $a,b,c$ are not all equal.
\begin{conjecture}[Uniform number-field $abc$]\label{conj:uabc}
For every $\eta>0$ there exists $c_\eta\ge0$, independent of $K$, such that for every number field $K$ and every $a,b,c\in K^\times$ with $a+b+c=0$,
\begin{equation}\label{eq:uabc}
h(a:b:c)\le(1+\eta)\left(r_K(a:b:c)+\frac{\log|\Delta_K|}{[K:\Q]}\right)+c_\eta.
\end{equation}
\end{conjecture}
This is the logarithmic form of Granville--Stark's formulation
$H(a,b,c)\ll_\eta(\operatorname{rd}_K N_K(a,b,c))^{1+\eta}$,
where $\operatorname{rd}_K=|\Delta_K|^{1/e}$ and $\log N_K=r_K$ \cite{GS}. Their discussion distinguishes unbounded-degree uniformity from the bounded-degree version supplied by the usual formulation of Vojta's general conjecture. Below, the conjecture is applied in fields whose degree grows with $p$, so a bounded-degree version would not suffice. The unspecified dependence of $c_\eta$ as $\eta\downarrow0$ matters in the final limiting argument.

For a monic integer polynomial $Q$ with roots $\rho_1,\ldots,\rho_D$, counted with multiplicity, put
\[
M(Q)=\prod_{i=1}^D\max\{1,|\rho_i|\},\qquad m(Q)=\log M(Q).
\]
If $P$ is monic irreducible of degree $e$ and $\rho$ is one of its roots, then $h(\rho)=m(P)/e$. Mahler measure is multiplicative, and
$M(Q)\le\sum_j|q_j|$ for $Q(X)=\sum_jq_jX^j$: integrate $\log|Q|$ on the unit circle and bound it by the logarithm of the coefficient length. In particular, a sum of $k$ distinct monomials has $m(Q)\le\log k$.

\section{A discriminant estimate with uniform constants}
\begin{lemma}\label{lem:discriminant}
Let $P\in\Z[X]$ be monic irreducible of degree $e$, let $\rho$ be a root, and let $K=\Q(\rho)$. Then
\begin{equation}\label{eq:discriminant}
\log|\Delta_K|\le e\log e+2(e-1)m(P).
\end{equation}
For $e=1$ both sides are $0$.
\end{lemma}
\begin{proof}
The polynomial discriminant satisfies
$|\operatorname{disc}P|=|\Delta_K|[\mathcal O_K:\Z[\rho]]^2$, so it suffices to bound $|\operatorname{disc}P|$. The Vandermonde matrix of the roots $\rho_i$ has row
$(1,\rho_i,\ldots,\rho_i^{e-1})$ of Euclidean norm at most
$\sqrt e\max(1,|\rho_i|)^{e-1}$. Hadamard's determinant inequality gives
\[
|\operatorname{disc}P|\le e^e M(P)^{2e-2}.
\]
Taking logarithms proves the claim; this is Mahler's classical discriminant inequality \cite{Mahler}.
\end{proof}

\section{Uniform \texorpdfstring{$abc$}{abc} and cyclic resultants}
Define, when nonzero,
\[
C_n(Q)=|\Res(Q,X^n-1)|=\prod_{Q(\rho)=0}|\rho^n-1|.
\]
The roots in every product or sum below are counted with multiplicity.

\begin{lemma}[An irreducible factor]\label{lem:factor}
Assume Conjecture~\ref{conj:uabc}. Let $P\in\Z[X]$ be monic irreducible with $P(0)=\pm1$, degree $e$, and $\mu=m(P)$. Suppose that $P$ is not cyclotomic. For every integer $n\ge1$,
\begin{equation}\label{eq:factor}
\log C_n(P)\ge\frac{n\mu}{1+\eta}-\log|\Delta_{\Q(\rho)}|-\frac{c_\eta e}{1+\eta}.
\end{equation}
\end{lemma}
\begin{proof}
Every root $\rho$ of $P$ is an algebraic unit: it is integral and its inverse is integral because $P(0)=\pm1$. Since $P$ is noncyclotomic, $\rho^n\ne1$. Apply~\eqref{eq:uabc} over $K=\Q(\rho)$ to
\[
(a,b,c)=(\rho^n,1-\rho^n,-1).
\]
At finite places all three entries are integral and the first and third are units. Thus the finite contribution to the projective height is zero, and at the archimedean places
\[
h(\rho^n:1-\rho^n:-1)\ge\frac1e\sum_{\sigma:K\hookrightarrow\C}\log\max(1,|\sigma(\rho)|^n)=\frac{n\mu}{e}.
\]
The radical is supported exactly on the prime ideals dividing $1-\rho^n$. Each such ideal has positive integral valuation there. Consequently
\[
e\,r_K(\rho^n:1-\rho^n:-1)
\le\log|\Norm_{K/\Q}(1-\rho^n)|=\log C_n(P).
\]
Multiplying~\eqref{eq:uabc} by $e$ and rearranging proves~\eqref{eq:factor}. The norm need not be squarefree; bounding its radical by its full absolute norm is sufficient.
\end{proof}

\begin{proposition}[Uniform polynomial estimate]\label{prop:resultant}
Assume Conjecture~\ref{conj:uabc}. Let $Q\in\Z[X]$ be monic, have $Q(0)=\pm1$, degree $D\ge1$, and $m=m(Q)$. If $C_n(Q)\ne0$, then
\begin{equation}\label{eq:resultant}
\log C_n(Q)\ge\frac{nm}{1+\eta}
-D\log(2D)-2Dm-c_\eta D.
\end{equation}
\end{proposition}
\begin{proof}
Factor $Q=\prod_iP_i^{w_i}$ into distinct monic irreducible integer polynomials. The integer constant terms of the factors multiply, with multiplicity, to $\pm1$, hence each is $\pm1$. Put $e_i=\deg P_i$ and $\mu_i=m(P_i)\ge0$. Then
\[
\sum_iw_ie_i=D,\qquad\sum_iw_i\mu_i=m.
\]
For noncyclotomic factors combine Lemmas~\ref{lem:factor} and~\ref{lem:discriminant}. Cyclotomic factors have $\mu_i=0$ and their nonzero resultants are integers of absolute value at least $1$, so they contribute nonnegatively and require no $abc$ application. Bound the accumulated errors by
\[
\sum_iw_ie_i\log e_i\le D\log(2D),\qquad
\sum_i2w_i(e_i-1)\mu_i\le2Dm,
\]
and $\sum_iw_ic_\eta e_i/(1+\eta)\le c_\eta D$. Multiplicativity of the resultant gives~\eqref{eq:resultant}. This argument permits reducibility and repeated roots.
\end{proof}

\section{From a resultant to one value: no symmetry required}
\begin{lemma}[Elementary transfer]\label{lem:transfer}
Let $Q$ be monic of degree $D$, and let $z^n=1$. Then
\begin{equation}\label{eq:transfer}
C_n(Q)\le n^D|Q(z)|M(Q)^{n-1}.
\end{equation}
\end{lemma}
\begin{proof}
For each root $\rho$ of $Q$, the identity
\[
\rho^n-1=\rho^n-z^n=(\rho-z)\sum_{j=0}^{n-1}\rho^{n-1-j}z^j
\]
gives
$|\rho^n-1|\le n|\rho-z|\max(1,|\rho|)^{n-1}$.
Multiply over all roots. Monicity gives $\prod|\rho-z|=|Q(z)|$ and the other product is $M(Q)^{n-1}$. Neither reciprocity nor a unit-circle condition on $\rho$ is used.
\end{proof}

\begin{corollary}\label{cor:value}
Under Conjecture~\ref{conj:uabc}, if $Q\in\Z[X]$ is monic, $Q(0)=\pm1$, $D=\deg Q\ge1$, $M(Q)\le\Lambda$, and $C_n(Q)\ne0$, then every $z^n=1$ satisfies
\begin{equation}\label{eq:value}
\log|Q(z)|\ge-\frac{\eta}{1+\eta}n\log\Lambda
-D\bigl(\log n+\log(2D)+2\log\Lambda+c_\eta\bigr).
\end{equation}
\end{corollary}
\begin{proof}
Subtract $D\log n+(n-1)m(Q)$ from the lower bound~\eqref{eq:resultant} using~\eqref{eq:transfer}. The leading term is
$-\eta nm(Q)/(1+\eta)+m(Q)$. Drop the nonnegative last term and use $m(Q)\le\log\Lambda$.
\end{proof}
In particular, for a family with bounded $M(Q)$, nonzero $C_n(Q)$, and
$D\log(2nD)=o(n)$, this gives $\log|Q(z)|\ge-o(n)$ uniformly. To justify the quantifiers, first fix $\eta$ sufficiently small; only afterwards let $n$ grow. No control on $c_\eta$ is needed for this order of limits.

\section{Compression of a fixed number of exponents}
\begin{lemma}\label{lem:compression}
Fix $k\ge2$ and let $p>k$ be prime. For any distinct $a_1,\ldots,a_k\in\Fp$, there are a monic polynomial
\[
Q(X)=\sum_{j=1}^kX^{b_j}\in\Z[X],\qquad
0=\min b_j<\max b_j=D,
\]
with distinct $b_j$, and a primitive $p$th root $z$, such that
\begin{equation}\label{eq:compressed-value}
\left|\sum_{j=1}^k\zeta_p^{a_j}\right|=|Q(z)|,
\qquad D=O_k\bigl(p^{1-1/(k-1)}\bigr).
\end{equation}
Moreover $Q(0)=1$, $Q(1)=k$, and $M(Q)\le k$.
\end{lemma}
\begin{proof}
Translate by $a_1$ and write $d_j=a_j-a_1$, so $d_1=0$. Let
\[
s=\left\lfloor(p-1)^{1/(k-1)}\right\rfloor,\qquad H=\lceil p/s\rceil.
\]
Divide each coordinate interval $\{0,\ldots,p-1\}$ into $s$ consecutive intervals of at most $H$ integers. The $p$ vectors
\[
(q d_2,\ldots,q d_k)\pmod p,\qquad 0\le q<p,
\]
lie in $s^{k-1}\le p-1$ boxes. Two vectors lie in the same box. Their parameter difference $t\in\Fp^\times$ has representatives $h_j$ of $td_j$ satisfying $|h_j|\le H$; put $h_1=0$. The $h_j$ are distinct, since multiplication by $t$ is invertible modulo $p$ and the $d_j$ are distinct. Set $h_*=\min_jh_j$, $b_j=h_j-h_*$, and $z=\zeta_p^{t^{-1}}$, where the inverse is taken modulo $p$. Then
\[
\sum_j\zeta_p^{a_j}=\zeta_p^{a_1}z^{h_*}\sum_jz^{b_j}.
\]
The prefactor has modulus $1$, $D\le2H$, and for fixed $k$,
$H=O_k(p^{1-1/(k-1)})$. The least and greatest exponents each occur once, so the constant and leading coefficients of $Q$ are both $1$. Its coefficient length is $k$.
\end{proof}

\begin{proof}[Proof of Theorem~\ref{thm:main}]
Use Lemma~\ref{lem:compression}. For fixed $k$ and all sufficiently large $p$, $D<p-1$. No primitive $p$th root can be a root of $Q$, because its minimal polynomial has degree $p-1$. Also $Q(1)=k\ne0$. Therefore $C_p(Q)\ne0$, including for any cyclotomic factors of $Q$. Corollary~\ref{cor:value} gives
\begin{equation}\label{eq:quantitative-main}
\log\left|\sum_j\zeta_p^{a_j}\right|
\ge-\frac{\eta}{1+\eta}p\log k
-O_{\eta,k}\bigl(p^{1-1/(k-1)}\log p\bigr).
\end{equation}
Given $\varepsilon>0$, choose fixed $\eta>0$ such that
$\eta\log k/(1+\eta)<\varepsilon/2$. The second term is $o(p)$ with a constant independent of the exponents. For sufficiently large $p$ it is at most $\varepsilon p/2$. This proves the theorem.
\end{proof}

\section{The symmetric five-term case and the aggregate norm target}\label{sec:fiveterm}
Return to the family $\alpha_p(a,b)$ of Section~\ref{sec:results}, with $p>5$ prime and $(a,b)$ admissible. Theorem~\ref{thm:main} gives $|\alpha_p(a,b)|\ge\exp(-\varepsilon p)$ uniformly for every fixed $\varepsilon>0$ and sufficiently large primes; that is, Conjecture~\ref{conj:B} holds under Conjecture~\ref{conj:uabc}.

Symmetry improves the compression degree, although it is not needed for the conditional conclusion. Applying the two-coordinate pigeonhole argument to $(a,b)$ gives representatives $ta,tb$ of size $O(\sqrt p)$. Replace them by their absolute values, reorder to $0<u<v$, and divide by their integer gcd, absorbing that gcd into the nonzero dilation $\lambda$. For large $p$ this gcd is less than $p$. We obtain $a\equiv\pm\lambda u$, $b\equiv\pm\lambda v$ and
\[
Q_{u,v}(X)=X^{2v}+X^{v+u}+X^v+X^{v-u}+1,
\qquad\gcd(u,v)=1,\quad v=O(\sqrt p),
\]
with $|\alpha_p(a,b)|=|Q_{u,v}(\zeta_p^{\lambda})|$. Its degree is $2v$ and $M(Q_{u,v})\le5$. Thus~\eqref{eq:value} has error $O_\eta(\sqrt p\log p)$ instead of the $O_\eta(p^{3/4}\log p)$ error for an arbitrary five-element set.

The same hypothesis also settles a stronger aggregate target. For a monic $Q$ define
\[
T_n(Q)=\sum_{Q(\rho)=0,\ |\rho|=1}\logplus\frac1{|\rho^n-1|},
\qquad \logplus x=\max(0,\log x).
\]
\begin{proposition}\label{prop:aggregate}
Under Conjecture~\ref{conj:uabc}, for monic $Q\in\Z[X]$ with $Q(0)=\pm1$, degree $D\ge1$, $M(Q)\le\Lambda$, and $C_n(Q)\ne0$,
\begin{equation}\label{eq:aggregate}
T_n(Q)\le\frac{\eta}{1+\eta}n\log\Lambda
+D\bigl(\log(2D)+2\log\Lambda+c_\eta+\log2\bigr).
\end{equation}
In particular $T_p(Q_{u,v})=o(p)$ uniformly for $v=O(\sqrt p)$.
\end{proposition}
\begin{proof}
For every root $\rho$, $|\rho^n-1|\le2\max(1,|\rho|)^n$. In $\log C_n(Q)$ remove the negative logarithms from unit-circle roots. Their total is $-T_n(Q)$; all remaining logarithms sum to at most $nm(Q)+D\log2$. Hence
\[
T_n(Q)\le nm(Q)+D\log2-\log C_n(Q).
\]
Insert~\eqref{eq:resultant} and use $m(Q)\le\log\Lambda$. For the last assertion choose $\eta$ first, then let $p$ grow, as in the proof of Theorem~\ref{thm:main}.
\end{proof}

\section{Limitations and future work}\label{sec:limits}

\paragraph{What is and is not proved.}
Theorem~\ref{thm:D} bounds the number of exceptional embeddings at the quasipolynomial scale, unconditionally, but leaves those embeddings uncontrolled. Theorem~\ref{thm:main} controls every embedding, but only at the subexponential scale and only under an arithmetic hypothesis. As noted in Section~\ref{sec:results}, neither Conjecture~\ref{conj:A} nor Conjecture~\ref{conj:B} follows unconditionally. An individual lower bound of the form $\exp\{-C(\log p)^B\}$ at every embedding in the symmetric five-term case remains unresolved by these arguments.

\paragraph{Scope and possible extensions.}
The polynomial estimates of Sections~\ref{sec:abc}--\ref{sec:fiveterm} hold for any monic integer polynomial with constant term $\pm1$, bounded Mahler measure, and nonzero cyclic resultant. The compression lemma explains why the distinct prime-order sums fit these hypotheses. Repeated terms can change the extreme coefficients, and general coefficients can produce nonunit roots. Composite orders also remove the automatic invertibility of the pigeonhole dilation. Those extensions of the conditional individual estimate require additional arguments and are not claimed. The unconditional exceptional-set estimate of Theorem~\ref{thm:D} uses a different argument and does allow composite conductors; the specific five-term symmetric family enters through the fixed curve $E$ of Section~\ref{sec:Dproof}, and the method should apply to other fixed real-analytic curves cut out by Laurent polynomials with nonvanishing real gradient, although we do not pursue this here.

\end{document}